\pdfoutput=1

\documentclass[12pt, a4paper]{amsart}

\usepackage[T1]{fontenc}						
\usepackage[utf8]{inputenc}						

\usepackage{fixltx2e}	

\usepackage{verbatim}	

\usepackage[italian, english]{babel}

\usepackage[babel=true, kerning=true, spacing=true, tracking=true]{microtype}

\usepackage[babel]{csquotes}
\usepackage[style=numeric, maxnames=5, hyperref, backend=bibtex]{biblatex}

\usepackage{braket}		
\usepackage{mathtools}	
\usepackage{amssymb}
\usepackage{amsthm}
\usepackage{amsfonts}
\usepackage[toc]{appendix}
\usepackage{bbm}
\usepackage[export]{adjustbox}
\usepackage[linewidth=1pt]{mdframed}
\usepackage{caption}
\usepackage{tabulary}
\usepackage{longtable}
\usepackage{array, ragged2e}

\DeclareMathOperator{\HS}{HS}
\DeclareMathOperator{\vspan}{span}
\DeclareMathOperator{\fchar}{char}
\DeclareMathOperator{\GL}{GL}
\DeclareMathOperator{\NF}{NF}
\newcommand{\blank}{{-}}

\theoremstyle{plain}
\newtheorem{thm}{Theorem}[section]
\newtheorem{lemma}[thm]{Lemma}
\newtheorem{prop}[thm]{Proposition}
\newtheorem{cor}[thm]{Corollary}
\theoremstyle{definition}
\newtheorem{defin}[thm]{Definition}

\theoremstyle{remark}
\newtheorem{rem}[thm]{Remark}
\newtheorem{notat}[thm]{Notation}

\renewcommand{\labelenumi}{(\roman{enumi})}

\def\cocoa{{\hbox{\rm C\kern-.13em o\kern-.07em C\kern-.13em o\kern-.15em A}}}

\usepackage{lmodern}
\normalfont

\usepackage{tabularx}
\usepackage{booktabs}	

\usepackage{tikz-cd}
\usetikzlibrary{arrows}

\usepackage{enumitem}

\setlist[itemize]{topsep=6.0pt plus 2.4pt minus 3.6pt, itemsep=3.0pt plus 1.5pt minus 0.6pt, parsep=3.0pt plus 1.5pt minus 0.6pt, leftmargin=29.37pt, listparindent=0.0pt, labelwidth=23.5pt}
\setlist[enumerate]{topsep=6.0pt plus 2.4pt minus 3.6pt, itemsep=3.0pt plus 1.5pt minus 0.6pt, parsep=3.0pt plus 1.5pt minus 0.6pt, leftmargin=29.37pt, listparindent=0.0pt, labelwidth=23.5pt}

\usepackage{hyperref}
\hypersetup{hidelinks, pdftitle={The Koszul property for spaces of quadrics of codimension three}, pdfauthor={Alessio D'Al\`i}}

\usepackage{pdfpages}

\numberwithin{equation}{section}

\begin{document}
	\author{Alessio D'Al\`i}
	\address{Dipartimento di Matematica, Universit\`a degli Studi di Genova, Via Dodecaneso 35, 16146 Genova, Italy}
	\email{dali@dima.unige.it}
	\title{The Koszul property for spaces of quadrics of codimension three}
	\date{\today}

\begin{abstract}
In this paper we prove that, if $\mathbbm{k}$ is an algebraically closed field of characteristic different from 2, almost all quadratic standard graded $\mathbbm{k}$-algebras $R$ such that $\dim_{\mathbbm{k}}R_2 = 3$ are Koszul. More precisely, up to graded $\mathbbm{k}$-algebra homomorphisms and trivial fiber extensions, we find out that only two (or three, when the characteristic of $\mathbbm{k}$ is $3$) algebras of this kind are non-Koszul.

Moreover, we show that there exist nontrivial quadratic standard graded $\mathbbm{k}$-algebras with $\dim_{\mathbbm{k}}R_1 = 4$, $\dim_{\mathbbm{k}}R_2 = 3$ that are Koszul but do not admit a Gr\"obner basis of quadrics even after a change of coordinates, thus settling in the negative a question asked by Conca.
\end{abstract}

\keywords{Koszul algebra, quadratic Gr\"obner basis, quadratic ideal, space of quadrics, Hilbert series}
\subjclass[2010]{16S37, 13P10, 13F20}

         \maketitle

\section{Introduction}
A (commutative) standard graded algebra $R$ over a field $\mathbbm{k}$ is a quotient of a polynomial ring over $\mathbbm{k}$ in a finite number of variables by a homogeneous ideal $I$ not containing any linear form. We say that $R$ is \emph{Koszul} when the minimal graded free resolution of $\mathbbm{k}$ as an $R$-module is linear: for a survey about Koszulness in the commutative setting, see \cite{CDRKoszul}. It is well-known that any Koszul algebra is quadratic, i.e. its defining ideal $I$ is generated by quadrics. Note that, if $R$ is a trivial fiber extension, i.e. it contains a linear form $\ell \in R_1$ such that $\ell R_1 = 0$, then $R$ is Koszul if and only if $R/{\ell}$ is Koszul.

For the rest of this introduction, let $R$ be a quadratic standard graded $\mathbbm{k}$-algebra, where $\mathbbm{k}$ is an algebraically closed field of characteristic different from two.
Backelin proved in \cite{BackelinThesis} that, if $\dim_\mathbbm{k}R_2 = 2$, then $R$ is Koszul (this actually holds for any field $\mathbbm{k}$). Conca then proved in \cite{Conca2009} that, if $\dim_\mathbbm{k}R_2 = 3$ and $R$ is Artinian, then $R$ is Koszul. The main problem we are addressing in this paper is to find out what happens when, in the latter case, we drop the Artinian assumption. We will prove in Theorem \ref{mainthm} that, up to trivial fiber extension, the only non-Koszul algebras live in embedding dimension three. More precisely, these non-Koszul algebras in three variables are the two (or three, when $\fchar\mathbbm{k} = 3$) objects exhibited by Backelin and Fr\"oberg in \cite{BackFrPoin}. As a byproduct, we shall also obtain a list of all possible Hilbert series when $\dim_\mathbbm{k}R_2 = 3$, see Theorem \ref{hilblist}.

Our results are in agreement with the numerical data presented by Roos in the characteristic 0 four-variable case, see \cite{Roos}.

Another interesting property one can investigate is whether or not $R$ is G-quadratic, i.e. it admits a Gr\"obner basis of quadrics (possibly after a change of coordinates). It is well-known that G-quadraticness is a sufficient but not necessary condition for Koszulness, see \cite[Section 6]{ERT}. We will show in Corollary \ref{notGquad} that, in our context, there exist four-variable Koszul algebras which are not G-quadratic and are not trivial fiber extensions: this answers a question asked by Conca in \cite[Section 4]{Conca2009}.

Computations made using the computer algebra system \cocoa~\cite{CoCoA} helped us to produce conjectures and gave us hints about the behaviour of the objects studied.

\section{Tools and techniques}
\subsection{Koszulness and related concepts}
We are going to recall very briefly some definitions to put our results into context. We direct the interested reader to the survey \cite{CDRKoszul} for further information.

Let $\mathbbm{k}$ be a field, $S$ the polynomial ring $\mathbbm{k}[x_1, \ldots, x_n]$ and $I$ a homogeneous ideal of $S$ not containing any linear form. Let $R$ be the standard graded $\mathbbm{k}$-algebra $S/I$.

\begin{defin}
The (commutative) $\mathbbm{k}$-algebra $R$ is \emph{Koszul} if the minimal graded free resolution of the residue field $\mathbbm{k}$ as an $R$-module is linear.
\end{defin}

It is usually very difficult to establish the Koszulness of a certain algebra by making use of the definition only. Because of this we will collect below some sufficient conditions, see Theorem \ref{suffcondkoszul}.

\begin{defin}
If there exists a family $\mathfrak{F}$ of ideals of $R$ such that: \begin{itemize}
\item $(0) \in \mathfrak{F}$, $(R_1) \in \mathfrak{F}$,
\item every nonzero $I \in \mathfrak{F}$ is generated by elements of degree 1,
\item for every nonzero $I \in \mathfrak{F}$ there exists $J \in \mathfrak{F}$ such that $J \subseteq I$, $I/J$ is cyclic and $J :_R I \in \mathfrak{F}$,
\end{itemize}
we say that $R$ admits a \emph{Koszul filtration} $\mathfrak{F}$, see \cite{CTVKoszul}.
\end{defin}

\begin{defin}
If there exists a change of coordinates $g \in \GL_n(\mathbbm{k})$ such that $g(I)$ admits a quadratic Gr\"obner basis with respect to some term order $\tau$, we say that $R$ is \emph{G-quadratic}.
\end{defin}

\begin{defin}An \emph{integral weight} $\omega = (\omega_1, \ldots, \omega_n)$ is an element of $\mathbb{N}^n$. Every integral weight induces a grading on $S$ obtained by imposing that $\deg(x_i) = \omega_i$ for every $i \in \{1, \ldots, n\}$. For this reason, we will often introduce a weight $\omega$ by associating an integer $\omega_i$ with each variable $x_i$ in $S$.

Given a nonzero $f \in S$, we will denote by $in_{\omega}(f)$ the part of $f$ having maximum degree with respect to the grading induced by $\omega$. We will denote by $in_{\omega}(I)$ the \emph{initial ideal} of $I$ with respect to $\omega$, i.e. the ideal generated by $\{in_{\omega}(f) \ | \ f \in I, f \neq 0\}$.
\end{defin}

\begin{rem} \label{initialrem}
In what follows we will deal very frequently with initial ideals, with respect to both weights and term orders: let us record an easy observation for further reference. Fix an ideal $I $ of $S$ and let $\omega$ be either a weight or a term order on $S$. Now pick an ideal $J$ of $S$ such that $J \subseteq in_{\omega}(I)$ (for instance, $J$ may be chosen as the ideal generated by the initial parts with respect to $\omega$ of a fixed generating set of $I$). If the Hilbert series $\mathbf{H}$ of $S/J$ happens to be coefficientwise smaller than or equal to the Hilbert series of $S/I$, then (denoting Hilbert series by $\HS$ and coefficientwise inequality by $\preceq$) one has that
\[\mathbf{H} \preceq \HS(S/I) = \HS(S/in_{\omega}(I)) \preceq \HS(S/J) = \mathbf{H}\]
and hence $in_{\omega}(I) = J$.
We will often make use of this remark in Section \ref{mainthmproof} when we claim what $in_{\omega}(I)$ is: for an example, see the end of Subsection \ref{firstcase}.
\end{rem}

\begin{thm}[Sufficient conditions for Koszulness] \label{suffcondkoszul}
One has that:
\begin{itemize}
\item if $R$ admits a Koszul filtration, then $R$ is Koszul.
\item if $R$ is G-quadratic, then $R$ is Koszul.
\item if there exists an integral weight $\omega \in \mathbb{N}^n$ such that $S/in_{\omega}(I)$ is Koszul, then $R$ is Koszul.
\end{itemize}
\end{thm}

\begin{defin}
If $R$ contains a nonzero linear form $\ell$ such that $\ell R_1 = 0$, then $R$ is called a \emph{trivial fiber extension} (of $R/{\ell}$).
\end{defin}

The concept of trivial fiber extension gives us the following useful reduction when looking for Koszulness or G-quadraticness.

\begin{prop}[\cite{BackFr}, {\cite[Lemma 4]{Conca2000}}] \label{KoszulTFE}
Assume $R$ contains a nonzero linear form $\ell$ such that $\ell R_1 = 0$. Then $R$ is Koszul (resp. G-quadratic) if and only if $R/{\ell}$ has the same property.
\end{prop}

\subsection{Apolarity pairing} \label{apolarity}
A very useful tool in what follows is the so-called \emph{apolarity} pairing, see for instance Geramita \cite[Lecture 2]{GeramitaApolarity} or Iarrobino and Kanev \cite[Chapter 1 and Appendix A]{IarrobinoKanev} for an introduction. We will recall here some results following Ehrenborg and Rota \cite[Section 3]{EhrenborgRotaApolarity} as our main reference.

Let $d$ be a positive integer and $\mathbbm{k}$ be a field of characteristic either 0 or greater than $d$. Let $S$ and $T$ be two copies of the polynomial ring over $\mathbbm{k}$ in $n$ variables ($x_1, \ldots, x_n$ and $u_1, \ldots, u_n$ respectively). We want to regard each $u_i$ as the partial derivative operator $\frac{\partial}{\partial x_i}$; more precisely, we define a symmetric bilinear pairing \[\langle \blank \mid \blank \rangle\colon T_d \times S_d \to \mathbbm{k}\] by asking that \[\langle u_1^{i_1} \ldots u_n^{i_n} \mid x_1^{j_1} \ldots x_n^{j_n} \rangle = i_1! \ldots i_n! \delta_{i_1, j_1} \ldots \delta_{i_n, j_n}\]
(where $\delta_{ij}$ equals 1 when $i = j$ and 0 otherwise) and then extending bilinearly.
Note that, because of our hypothesis on the characteristic of $\mathbbm{k}$, this pairing is nonsingular.
It is also known (\cite[Corollary 3.1]{EhrenborgRotaApolarity}) that acting on a vector subspace $V \subseteq S_d$ by a change of coordinates $g \in \GL_n(\mathbbm{k})$ corresponds to acting on the orthogonal $V^{\perp} \subseteq T_d$ by the change of coordinates $\tilde{g}$ obtained by inverting and transposing the matrix representing $g$ (and analogously if we take a vector subspace of $T_d$ instead).

The proof of the following result is a simple computation that works for any field of characteristic either 0 or greater than $d$, see for instance \cite[Proposition 3.1]{EhrenborgRotaApolarity}. The dot stands for the usual dot product, whereas $\mathbf{u}$ and $\mathbf{x}$ stand for $(u_1, \ldots, u_n)$ and $(x_1, \ldots, x_n)$ respectively.

\begin{lemma} \label{EhrenborgRotaLemma}
Given $\mathbf{c} \in \mathbbm{k}^n$, one has that for any $f \in S_d$ \[\langle (\mathbf{c} \cdot \mathbf{u})^d \mid f(\mathbf{x}) \rangle = d! f(\mathbf{c})\] and, analogously, for any $g \in T_d$ \[\langle g(\mathbf{u}) \mid (\mathbf{c} \cdot \mathbf{x})^d \rangle = d!g(\mathbf{c}).\]
\end{lemma}

\begin{notat}
Let $V \subseteq S_d$ be a vector subspace and let $I = (V)$ be the ideal generated by it. For the rest of this paper we will denote the projective variety which is the zero locus of $I$ either by $\mathcal{V}(I)$ or $\mathcal{V}(V)$.
\end{notat}

We will be interested in studying the nonzero linear forms of $S/I$ whose $d$-th power is zero, or equivalently the nonzero linear forms of $S$ whose $d$-th power lies in $V$. Because of this equivalence, we will consider the two concepts interchangeably.

\begin{defin}
Let $\ell \in (S/I)_1$, $\ell \neq 0$.
\begin{itemize}
\item The \emph{rank} of $\ell$ is the dimension of $\ell (S/I)_1$ as a $\mathbbm{k}$-vector space.
\item If $\ell^2 = 0$ in $S/I$, we will call $\ell$ a \emph{null-square linear form}. (Note that a null-square linear form is nonzero by definition.)
\end{itemize}
\end{defin}

Lemma \ref{EhrenborgRotaLemma} immediately implies the following very useful statement:
\begin{lemma}  \label{linear_forms_and_points}
Let $V$ be a vector subspace of $S_d$. The correspondence \[\mathbb{P}(S_1) \ni [\mathbf{c} \cdot \mathbf{x}] \mapsto \mathbf{[c]} \in \mathbb{P}^{n-1}_{\mathbbm{k}}\] restricts to a bijection between the nonzero linear forms of $S$ whose $d$-th power lies in $V$ and the points of $\mathcal{V}(V^{\perp})$. Symmetrically, we also get a bijection between the nonzero linear forms of $T$ whose $d$-th power lies in $V^{\perp}$ and the points of $\mathcal{V}(V)$.
\end{lemma}

In the rest of the paper, when no confusion arises, we will use the same name for the two copies of the polynomial ring in the pairing.

We end this section by a technical remark for further reference.

\begin{rem} \label{apolarityrem}
Let $\{G_1, \ldots, G_c\}$ be a $\mathbbm{k}$-basis of $V^{\perp}$. If there exist monomials $m_1, \ldots, m_c$ such that each $m_i$ appears in the support of $G_i$ only, then:
\begin{itemize}
\item there exists a $\mathbbm{k}$-basis of $V$ (and hence a minimal generating set for $I$) made of elements of the form \[F_m = m - a_1m_1 - \ldots - a_cm_c,\] where $m$ is a monomial different from $m_1, \ldots, m_c$ and the coefficient $a_i \in \mathbbm{k}$ is nonzero if and only if $m$ appears in $G_i$;
\item the cosets of the $m_i$'s form a $\mathbbm{k}$-basis of $(S/I)_d$. 
\end{itemize}
\end{rem}

\section{The main result} \label{mainresult}
We state below the main result of this paper and will devote Section \ref{mainthmproof} to proving it. For the rest of the section, when $S$ is a polynomial ring, we will denote by $\mathfrak{m}$ the maximal ideal generated by its variables.
\begin{thm} \label{mainthm}
Let $\mathbbm{k}$ be an algebraically closed field of characteristic different from 2, let $S$ be a polynomial ring over $\mathbbm{k}$ and let $I \subseteq \mathfrak{m}^2$ be a quadratic ideal of $S$. If $\dim_{\mathbbm{k}}(S/I)_2 = 3$, then $S/I$ is Koszul if and only if it is not isomorphic as a graded $\mathbbm{k}$-algebra (up to trivial fiber extension) to any of these:
\begin{enumerate}
\item $\mathbbm{k}[x, y, z]/(y^2+xy, xy+z^2, xz)$
\item $\mathbbm{k}[x, y, z]/(y^2, xy+z^2, xz)$
\item $\mathbbm{k}[x, y, z]/(y^2, xy+yz+z^2, xz)$.
\end{enumerate}
Note that, if the characteristic of $\mathbbm{k}$ is not 3, (ii) and (iii) are actually isomorphic as graded $\mathbbm{k}$-algebras.
\end{thm}

As a byproduct of our proof of Theorem \ref{mainthm} and of previous work by Backelin and Fr\"oberg \cite[proof of Theorem 1]{BackFrPoin} and Conca \cite[Theorem 1.1]{Conca2009}, we shall also get a classification of all possible Hilbert series for standard graded quadratic $\mathbbm{k}$-algebras $R$ such that $\dim_{\mathbbm{k}}R_2 = 3$, see Theorem \ref{hilblist} below.

\begin{rem} \label{ShortArtinian}
Backelin and Fr\"oberg \cite[proof of Theorem 1 and Appendix]{BackFrPoin} gave without proof a classification (in all characteristics) of standard graded quadratic algebras $R$ such that $\dim_{\mathbbm{k}}R_1 = \dim_{\mathbbm{k}}R_2 = 3$ and $R$ is not a complete intersection. Unfortunately, in characteristic 2 their list is incomplete, since at least one (non-Koszul) case is missing, see (b) below. However, when the characteristic is different from 2, we proved that their list is indeed correct: this will be the object of a later arXiv note.

Under these hypotheses, the only non-Koszul algebras in characteristic different from 2 are the ones denoted by (i), (ii) and (iii) in our Theorem \ref{mainthm}. Two further remarks have to be made:
\begin{itemize}
\item[(a)] if $\fchar \mathbbm{k} \neq 3$, then (ii) and (iii) are isomorphic as graded $\mathbbm{k}$-algebras: one goes from (iii) to (ii) by the change of coordinates sending $x$ into $x-y/9-z/3$ and $z$ into $-y/3+z$. On the other hand, if $\fchar \mathbbm{k} = 3$, then (ii) and (iii) are not isomorphic. For instance, consider the polynomials $P_{\text{(ii)}}$ and $P_{\text{(iii)}}$ defining the two projective hypersurfaces given by $\{[\alpha, \beta, \gamma] \in \mathbb{P}^2_{\mathbbm{k}} \mid \alpha Q_1 + \beta Q_2 + \gamma Q_3\text{ has rank} \leq 2\}$, where the $Q_i$'s are the generators of the defining ideal of respectively (ii) and (iii). Since the characteristic is $3$, the two vector spaces generated by the first partial derivatives of respectively $P_{\text{(ii)}}$ and $P_{\text{(iii)}}$ have different dimensions and thus (ii) and (iii) cannot be isomorphic.  

\item[(b)] in \cite{BackFrPoin} one finds the algebra (i') $\mathbbm{k}[x, y, z]/(xz, xy+yz, x^2+z^2+xy)$, which is seen to be isomorphic to (i) as a graded $\mathbbm{k}$-algebra when $\fchar\mathbbm{k} \neq 2$. If $\fchar\mathbbm{k}=2$, though, (i') admits no null-square linear forms, whereas $y+z$ is a null-square linear form for (i): hence, the two algebras are not isomorphic in characteristic 2, but the algebra (i) is not featured in \cite{BackFrPoin}.
\end{itemize}
\end{rem}

\begin{thm} \label{hilblist}
Let $\mathbbm{k}$ be an algebraically closed field of characteristic different from 2, let $S$ the polynomial ring $\mathbbm{k}[x_1, \ldots, x_n]$ and $I \subseteq \mathfrak{m}^2$ a quadratic homogeneous ideal of $S$. If $\dim_{\mathbbm{k}}(S/I)_2 = 3$, these are the possible Hilbert series for $S/I$:

\vspace{5pt}
\renewcommand\arraystretch{1.4}
\begin{center}
    \begin{tabular}{| l | l | l}
    \cline{1-2}
    \multicolumn{1}{|c|}{\emph{Rational form}} & \multicolumn{1}{|c|}{\emph{Hilbert function}} & \\
    \cline{1-2}
    $1+nz+3z^2$                                 & $1~~n~~3$ & $(n \geq 4)$ \\ 
    $1+nz+3z^2+z^3$                            & $1~~n~~3~~1$ & $(n \geq 3)$ \\
    $(1+(n-1)z+(3-n)z^2-2z^3)/(1-z)$           & $1~~n~~3~~1~~1~~1~~1~~1\ldots$ & $(n \geq 3)$ \\
    $(1+(n-1)z+(3-n)z^2-z^3)/(1-z)$            & $1~~n~~3~~2~~2~~2~~2~~2\ldots$ & $(n \geq 3)$ \\
    $(1+(n-1)z+(3-n)z^2)/(1-z)$                 & $1~~n~~3~~3~~3~~3~~3~~3\ldots$ & $(n \geq 3)$ \\
    $(1+(n-2)z+(4-2n)z^2+(n-2)z^3)/(1-z)^2$ & $1~~n~~3~~4~~5~~6~~7~~8\ldots$ & $(n \geq 2)$ \\
    \cline{1-2}
    \end{tabular}
\end{center}
\vspace{5pt}
\renewcommand\arraystretch{1}
\end{thm}

\begin{rem}
Theorem \ref{hilblist} agrees with the list obtained, when $n$ equals 4 and $\fchar\mathbbm{k} = 0$, by Roos \cite[Appendix 1, Table 1]{Roos}. The Hilbert series of Theorem \ref{hilblist} are denoted by $H_7$ to $H_{12}$ there.
\end{rem}

\subsection{An outline of the proof of Theorem \ref{mainthm}}
\begin{table}[h!]
\caption{An outline of the proof strategy for Theorem \ref{mainthm}}
\label{outline}
\begin{mdframed}
Let $\mathbbm{k}$ be an algebraically closed field of characteristic different from 2 and let $S = \mathbbm{k}[x_1, \ldots, x_n]$, where $n \geq 4$. Let $I \subseteq \mathfrak{m}^2$ be a quadratic ideal of $S$ such that $\dim_{\mathbbm{k}}(S/I)_2 = 3$ and $S/I$ is not a trivial fiber extension.

\begin{itemize}
\item If the variety $\mathcal{V}(I)$ set-theoretically consists of two points, then $S/I$ is Koszul (Section \ref{twopoints}). 
\item If the variety $\mathcal{V}(I)$ set-theoretically consists of a point and there exists a null-square linear form of rank 2 in $S/I$, then $S/I$ is Koszul (Section \ref{onepointrank2}).
\item If the variety $\mathcal{V}(I)$ set-theoretically consists of a point and all null-square linear forms in $S/I$ have rank 1, then $S/I$ is Koszul (Section \ref{onepointrank1}).
\end{itemize}
\end{mdframed}
\end{table}

\begin{prop}
To prove Theorem \ref{mainthm}, it is enough to prove the statements in Table \ref{outline}.
\end{prop}

\begin{proof}
Let $R$ be the standard graded $\mathbbm{k}$-algebra $S/I$. We want to investigate Koszulness when $\dim_{\mathbbm{k}}R_2 = 3$. Conca proved in \cite{Conca2009} that, if we further ask that $R$ is Artinian, we always get a Koszul algebra (moreover, we get a G-quadratic algebra ``almost always''). He also gave a classification of the general case when $\dim_{\mathbbm{k}}R_1 \leq 3$, showing that, at least in characteristic 0, the only non-Koszul algebras are the ones appearing in the statement of Theorem \ref{mainthm}. This last statement holds also in characteristic greater than 2: see Remark \ref{ShortArtinian} above for more details.

We are thus interested in finding out what happens for four or more variables in the non-Artinian case. By Proposition \ref{KoszulTFE}, we can further assume we are dealing with an algebra which is not a trivial fiber extension.

Let us assume $S$ contains the variables $x, y, z, t$ and let us consider $\mathcal{V}(I)$. We know that, since $R$ is not Artinian, this variety is not empty: without loss of generality and recalling Lemma \ref{linear_forms_and_points}, we can assume $x^2$ is in $V^{\perp}$. If the variety contains another point, we can further assume $y^2$ is in $V^{\perp}$. What would happen if a third point were present? If the point did not lie on the line spanned by the first two points, then we could assume $z^2$ were in $V^{\perp}$, but then $V^{\perp}$ would be generated by $x^2$, $y^2$ and $z^2$ and hence $R$ would be a trivial fiber extension, since $tR_1$ would be zero. On the other hand, if the third point were in fact on the line, then $xy$ would be in $V^{\perp}$ and $R$ would again be a trivial fiber extension. Therefore, the variety $\mathcal{V}(I)$ set-theoretically consists of either one or two points.

By \cite[Lemma 3]{Conca2000}, we know that there exists at least one null-square linear form in $R$. Since we are excluding trivial fiber extensions, we cannot find any null-square linear form in $R$ whose rank is 0. Moreover, a null-square linear form of rank 3 can exist only if $R$ is Artinian by \cite[Lemma 2]{Conca2000}.
\end{proof}

Before starting the proof of Theorem \ref{mainthm}, we insert here a table of algebras that will come in handy in what follows. In Table \ref{koszultable} and in the rest of the paper we will denote by $\mathbf{A}_n$ the Hilbert series 
\[\frac{1+(n-1)z+(3-n)z^2-2z^3}{1-z} = 1 + nz + 3z^2 + z^3 + z^4 + \ldots\]
and by $\mathbf{B}_n$ the Hilbert series
\[\frac{1+(n-1)z + (3-n)z^2 - z^3}{1-z} = 1 + nz + 3z^2 + 2z^3 + 2z^4 + \ldots\]

\renewcommand\arraystretch{1.2}
\begin{table}[h]
\centering
\caption{Some useful algebras. All ideals are considered into the smallest polynomial ring $S$ over $\mathbbm{k}$ which contains their variables. Since all ideals in the table are unital binomial ideals (see Appendix \ref{appendixkoszul}), their Hilbert series do not depend on the characteristic of $\mathbbm{k}$. Note that (10) is not unital binomial itself but, if $\fchar\mathbbm{k} \neq 2$, it becomes so after a suitable change of coordinates.}
\label{koszultable}
    \begin{tabular}{|c|l|c|}
    \hline
       \#    & \multicolumn{1}{|c|}{Generators of $I$} & $\HS(S/I)$ \\
    \hline
        (1) & $t_2^2, t_3^2, t_2t_3, xy, xt_1, xt_3, yt_1, yt_2, t_1t_2, t_1t_3, xt_2-t_1^2, yt_3-t_1^2$ & $\mathbf{B}_5$\\
    (2) & $t_2^2, xt_1, t_1t_2, t_1^2-xt_2, xy, yt_1, yt_2$ & $\mathbf{B}_4$ \\
    (3) & $t_1^2, t_2^2, t_1t_2, yt_1, xt_2, xy, xt_1-yt_2$ & $\mathbf{B}_4$ \\
    (4) & $t^2, y_2t, xy_1, xy_2, y_1^2, y_1y_2, y_2^2-xt$ & $\mathbf{A}_4$ \\
    (5) & $t^2, y_2t, xy_1, xy_2, y_1^2-xt, y_1y_2, y_2^2-y_1t$ & $\mathbf{A}_4$ \\
    (6) & $t^2, y_2t, y_3t, xy_1, xy_2, xy_3, y_1^2-xt, y_1y_2, y_1y_3, y_2^2, y_2y_3-y_1t, y_3^2$ & $\mathbf{A}_5$ \\
    (7) & $t^2, y_2t, xy_1, xy_2, y_1^2-xt, y_1t, y_2^2$ & $\mathbf{A}_4$ \\
    (8) & $t^2, y_2t, xy_1, xy_2-y_1t, y_1^2-xt, y_1y_2, y_2^2$ & $\mathbf{A}_4$ \\
    (9) & $xy, xt_1, xt_2-y^2, xt_3-t_1^2, yt_1, yt_2, yt_3, t_1t_2, t_1t_3, t_2^2, t_2t_3, t_3^2$ & $\mathbf{A}_5$ \\
    (10) & $xy, xt_1, xt_2-y^2-t_1^2, yt_1, yt_2, t_1t_2, t_2^2$ & $\mathbf{A}_4$ \\
    (11) & $xy-t_1^2, xt_1, xt_2-y^2, yt_1, yt_2, t_1t_2, t_2^2$ & $\mathbf{A}_4$ \\
    (12) & $xy, xt_1-y^2, xt_2-t_1^2, xt_3, yt_1, yt_2, yt_3-t_1^2, t_1t_2, t_1t_3, t_2^2, t_2t_3, t_3^2$ & $\mathbf{A}_5$ \\
    (13) & $xy, y^2-xt_1, yt_1, yt_2, t_1^2, t_1t_2, t_2^2$       & $\mathbf{B}_4$ \\
    \hline
    \end{tabular}
\end{table}
\renewcommand\arraystretch{1}

\begin{lemma} \label{allkoszul}
All algebras in Table \ref{koszultable} admit a Koszul filtration and hence are Koszul.
\end{lemma}

\begin{proof} 
See Appendix \ref{appendixkoszul}.
\end{proof}

We are now ready to begin.

\section{Proof of Theorem \ref{mainthm}} \label{mainthmproof}
The aim of this section is to prove the statements contained in Table \ref{outline}. Let us set the notation for the rest of the section.
 
Let $\mathbbm{k}$ be an algebraically closed field of characteristic different from 2, let $S$ be a polynomial ring over $\mathbbm{k}$ in $n$ variables and let $I$ be a quadratic homogeneous ideal of $S$ not containing any linear form. Let $R$ be the quotient algebra $S/I$. We will denote by $V$ the vector subspace of $S_2$ generated by the quadrics of $I$ and by $V^{\perp}$ its orthogonal with respect to the apolarity pairing defined in Section \ref{apolarity}. Lowercase Latin letters different from $x$, $y$, $z$ and $t$ (possibly with indices) will be used to denote elements of $\mathbbm{k}$; moreover, we will use ``*'' to denote coefficients in $\mathbbm{k}$ whose value plays no role in the discussion.

\subsection{The variety $\mathcal{V}(I)$ set-theoretically consists of two points} \label{twopoints}
Let $S = \mathbbm{k}[x, y, t_1, \ldots, t_{n-2}]$. We can assume without loss of generality that \[\mathcal{V}(I) = \{[1, 0, 0, \ldots, 0], [0, 1, 0, \ldots, 0]\}\] and hence $x^2$ and $y^2$ do not appear in any generator of $I$.
Then $I \subseteq (xy) + (t_1, \ldots, t_{n-2})S_1$ and, since we also know that $I$ is generated by $n-3$ quadrics, we have that the Hilbert series of $S/I$ is coefficientwise greater than or equal to $\mathbf{B}_n$. Moreover, \[V^{\perp} = \vspan\{x^2, y^2, xL_1 + yL_2 + Q(t_1, \ldots, t_{n-2})\},\] where $L_1$ is a linear form not containing $x$, $L_2$ is a linear form not containing $y$ and $Q$ is a quadratic form in the remaining variables $t_1, \ldots, t_{n-2}$. We are going to distinguish some cases according to the rank of the quadric $Q$ (call this number $r$).

\subsubsection{The rank of $Q$ is greater than or equal to 2} \label{firstcase}
After a suitable change of coordinates we can assume $Q = 2t_1t_2 + t_3^2 + \ldots + t_r^2$.
Then, applying Remark \ref{apolarityrem}, the generators of $I$ are
\begin{equation*}
\begin{split}
&t_1^2, t_2^2, xy-*t_1t_2,\\
&xt_i-*t_1t_2, yt_i-*t_1t_2\textrm{ for all }i \in \{1, \ldots, n-2\},\\
&t_i^2 - t_1t_2\textrm{ for all }i \in \{3, \ldots, r\},\\
&t_i^2 \textrm{ for all }i \in \{r+1, \ldots, n-2\},\\
&t_it_j\textrm{ for all }i, j \in \{1, \ldots, n-2\}\textrm{ such that }i < j, (i, j) \neq (1, 2).
\end{split}
\end{equation*}

Let $\tau$ be a term order such that $t_1 <_{\tau} t_2 <_{\tau} v$ for any variable $v$ different from $t_1$ and $t_2$. Then the ideal \[J := (t_3, \ldots, t_{n-2})S_1 + (t_1^2, t_2^2, xy, xt_1, xt_2, yt_1, yt_2)\] is contained in $in_{\tau}(I)$ and, since the Hilbert series of $S/J$ is $\mathbf{B}_n$, in fact $in_{\tau}(I) = J$. (Note that this is precisely the setting of Remark \ref{initialrem}.) Hence, $R$ is G-quadratic.

\subsubsection{The rank of $Q$ is 1}
After a suitable change of coordinates we can assume $Q = t_1^2$. Since we want to exclude trivial fiber extensions, at least one of $L_1$ and $L_2$ must contain $t_2$: without loss of generality, let $L_1$ contain $t_2$. Then, after a change of coordinates, we can assume $L_1 = 2t_2$.

\begin{itemize}
\item Assume $S$ has more than four variables: then we can further assume $L_2 = 2t_3$. On the other hand, $S$ cannot have more than five variables, otherwise $R$ would be a trivial fiber extension: hence \[I = (t_2^2, t_3^2, t_2t_3, xy, xt_1, xt_3, yt_1, yt_2, t_1t_2, t_1t_3, xt_2-t_1^2, yt_3-t_1^2).\] Since $R$ is the algebra (1) in Table \ref{koszultable}, it is Koszul by Lemma \ref{allkoszul}.

\item If $S$ has exactly four variables, then $V^{\perp} = \vspan \{x^2, y^2, 2xt_2 + yL_2 + t_1^2\}$, hence \[I = (t_2^2, xt_1, t_1t_2, t_1^2 - xt_2, xy-*xt_2, yt_1-*xt_2, yt_2 -*xt_2).\]
Let $\omega = \begin{pmatrix} x & y & t_1 & t_2 \\ 2 & 3 & 1 & 0 \end{pmatrix}$ be a weight on $S$; then \[in_{\omega}(I) = (t_2^2, xt_1, t_1t_2, t_1^2-xt_2, xy, yt_1, yt_2)\] and, since $S/in_{\omega}(I)$ is the Koszul algebra (2) in Table \ref{koszultable}, $S/I$ is Koszul as well.
\end{itemize}

\subsubsection{The quadric $Q$ equals zero}
Since we want to exclude trivial fiber extensions, at least one of $L_1$ and $L_2$ must contain $t_1$: without loss of generality, let $L_1$ contain $t_1$. After a suitable change of coordinates we get $L_1 = t_1$. For analogous reasons we can repeat the procedure in order to obtain $L_2 = t_2$, hence $V^{\perp} = \vspan \{x^2, y^2, xt_1+yt_2\}$. The ring $S$ must contain exactly four variables, otherwise $R$ would be a trivial fiber extension: therefore, \[I = (t_1^2, t_2^2, t_1t_2, yt_1, xt_2, xy, xt_1-yt_2).\]
This gives us the Koszul algebra (3) in Table \ref{koszultable}.

This ends the discussion of the case when $\mathcal{V}(I)$ consists of two points. \qed

\subsection{The variety $\mathcal{V}(I)$ consists of a point and $R$ has a null-square linear form of rank 2} \label{onepointrank2}
\begin{table}[h]
\centering
\caption{Integral weights used in Section \ref{onepointrank2}}
\label{weighttable}
    \begin{tabular}{|c|c|c|c|c|c|c|}
    \hline
     & $x$ & $y_1$ & $y_2$ & $y_3$ & $y_i$ (for $4 \leq i \leq n-2$) & $t$ \\ \hline
    $\omega^{(1)}$ & 4 & 3    & 2    & 4    & 4                                              & 0 \\
    $\omega^{(2)}$ & 9 & 5    & 3    & 8    & 8                                              & 0 \\
    $\omega^{(3)}$ & 7 & 4    & 3    & 3    & 5                                              & 0 \\
    $\omega^{(4)}$ & 4 & 2    & 1    & 2    & 2                                              & 0 \\
    $\omega^{(5)}$ & 4 & 2    & 1    & 1    & 2                                              & 0 \\
    $\omega^{(6)}$ & 5 & 3    & 0    & 2    & 2                                              & 1 \\
    $\omega^{(7)}$ & 3 & 2    & 0    & 1    & 1                                              & 1 \\ \hline
    \end{tabular}
\end{table}
Let $S = \mathbbm{k}[x, y_1, \ldots, y_{n-2}, t]$. After a change of coordinates we can assume $\mathcal{V}(I) = \{[1, 0, \ldots, 0]\}$ and hence, since $I \subseteq (y_1, \ldots, y_{n-2}, t)S_1$ and $I$ is generated by $n-3$ quadrics, the Hilbert series of $S/I$ is coefficientwise greater than or equal to $\mathbf{A}_n$.

Without loss of generality, let $t$ be a rank 2 null-square linear form. Then at least one of $\overline{y_1t}, \ldots, \overline{y_{n-2}t} \in R_2$ is nonzero: from now on, assume $\overline{y_1t}$ is nonzero.\\

Suppose there exists $i \in \{2, \ldots, n-2\}$ such that $\dim_{\mathbbm{k}}\vspan\{\overline{y_1t}, \overline{y_it}\} = 2$ and let $i = 2$ without loss of generality. Then $I$ is generated by $t^2$ and $m - *y_1t - *y_2t$, where $m$ varies over the monomials of $S_2$ different from $x^2$, $y_1t$, $y_2t$ and $t^2$.  Let $\tau$ be a DegRevLex term order on $S$ such that $t <_{\tau} y_1 <_{\tau} y_2 <_{\tau} v$ for all $v \in \{x, y_3, \ldots, y_{n-2}\}$; one checks that \[in_{\tau}(I) = (y_3, \ldots, y_{n-2})S_1 + (t^2, xt, xy_1, xy_2, y_1y_2, y_1^2, y_2^2)\] and hence $R$ is G-quadratic.\\

From now on, we shall assume that \begin{center}\fbox{$\dim_{\mathbbm{k}}\vspan\{\overline{y_1t}, \overline{y_it}\} = 1$ for all $i \in \{2, \ldots, n-2\}$.}\end{center} Then $\dim_{\mathbbm{k}}\vspan\{\overline{y_1t}, \overline{xt}\} = 2$ and the generators of $I$ are \[\begin{split}&t^2,\\&y_it-a_iy_1t\textrm{ for all }i \in \{2, \ldots, n-2\},\\&xy_j-b_jxt-*y_1t\textrm{ for all }j \in \{1, \ldots, n-2\},\\&y_ky_{\ell}-*xt-*y_1t\textrm{ for all }k, \ell \in \{1, \ldots, n-2\}\textrm{ with }k \leq \ell.\end{split}\]
By applying the change of coordinates \[\left\{ \begin{array}{lll} x & \mapsto & x \\ y_1 & \mapsto & y_1+b_1t \\ y_i & \mapsto & y_i+a_iy_1+b_it\textrm{ for all }i \in \{2, \ldots, n-2\} \\ t & \mapsto & t \end{array} \right.\] one gets the ideal $I'$ generated by

\begin{equation}\tag{$\clubsuit$} \label{Iprimegens}
\begin{split}&t^2,\\&y_it\textrm{ for all }i \in \{2, \ldots, n-2\},\\&xy_j-d_{0j}y_1t\textrm{ for all }j \in \{1, \ldots, n-2\},\\&y_ky_{\ell}-c_{k{\ell}}xt-d_{k{\ell}}y_1t\textrm{ for all }k, \ell \in \{1, \ldots, n-2\}\textrm{ with }k \leq \ell.\end{split}\end{equation}

Note that $R$ is Koszul (resp. G-quadratic) if and only if $S/I'$ has the same property.

If $xt$ does not appear in \eqref{Iprimegens}, then $I' \subseteq (t^2) + (y_1, \ldots, y_{n-2})S_1$ and, since we also know that $I'$ is generated by $n-3$ quadrics, the Hilbert series of $S/I'$ is coefficientwise greater than or equal to $\mathbf{B}_n$. In this case, let $\tau$ be a term order on $S$ such that $t < y_1 < v$ for all $v \in \{x, y_2, \ldots, y_{n-2}\}$. Then one has that \[in_{\tau}(I') = (y_2, \ldots, y_{n-2})S_1 + (t^2, xy_1, y_1^2)\] and hence $R$ is G-quadratic.

From now on assume instead that $xt$ appears in \eqref{Iprimegens}, i.e. \begin{center}\fbox{$c_{k{\ell}}$ is nonzero for some choice of $k$ and $\ell$.}\end{center}

\begin{itemize}
\item Suppose there exists $i \in \{2, \ldots, n-2\}$ such that $c_{ii} \neq 0$. Let $i = 2$ and $c_{22} =1$ without loss of generality (one can simply rescale $y_2$ to obtain the latter). Then, using the weight $\omega^{(1)}$ (see Table \ref{weighttable}), one has that \[in_{\omega^{(1)}}(I') = (y_3, \ldots, y_{n-2})S_1 + (t^2, y_2t, xy_1, xy_2, y_1^2, y_1y_2, y_2^2-xt)\]
and, since $S/in_{\omega^{(1)}}(I')$ is the Koszul algebra (4) in Table \ref{koszultable}, $S/I'$ is Koszul as well.

\item Suppose $c_{ii} = 0$ for all $i \in \{2, \ldots, n-2\}$ and suppose there exists $j \in \{2, \ldots, n-2\}$ such that $c_{1j} \neq 0$. Without loss of generality, let $c_{12} = 1$. By completing the weight $\omega^{(2)}$ to a term order $\tau$, one notes that
\[in_{\tau}(I') = (y_3, \ldots, y_{n-2})S_1 + (t^2, y_2t, xy_1, xy_2, y_1^2, xt, y_2^2)\]
and hence $R$ is G-quadratic.

\item Suppose $c_{ii} = c_{1i} = 0$ for all $i \in \{2, \ldots, n-2\}$ and suppose there exist $k, \ell \in \{2, \ldots, n-2\}$ such that $k < \ell$ and $c_{k{\ell}} \neq 0$. Without loss of generality, let $c_{23} = 1$. By completing the weight $\omega^{(3)}$ to a term order $\tau$, one gets that
\[in_{\tau}(I') = (y_4, \ldots, y_{n-2})S_1 + (t^2, y_2t, y_3t, xy_1, xy_2, xy_3, y_1^2, y_1y_2, y_1y_3, y_2^2, xt, y_3^2)\]
and hence $R$ is G-quadratic.
\end{itemize}

From now on we will assume (without loss of generality) that \begin{center}\fbox{$c_{11} = 1$ and $c_{k{\ell}} = 0$ for all $k, \ell$ such that $k \leq \ell$, $(k, \ell) \neq (1, 1)$.}\end{center} Note that, in order to avoid trivial fiber extensions, for all $i \in \{2, \ldots n-2\}$ at least one of $d_{0i}, d_{1i}, \ldots, d_{{n-2,} i}$ must be nonzero.
\begin{itemize}

\item Suppose there exists $i \in \{2, \ldots, n-2\}$ such that $d_{ii} \neq 0$. Let $i = 2$ and $d_{22} =1$ without loss of generality. One has that
\[in_{\omega^{(4)}}(I') = (y_3, \ldots, y_{n-2})S_1 + (t^2, y_2t, xy_1, xy_2, y_1^2-xt, y_1y_2, y_2^2-y_1t)\]
and, since $S/in_{\omega^{(4)}}(I')$ is Koszul (being a trivial fiber extension of the Koszul algebra (5) in Table \ref{koszultable}), also $S/I'$ is.

\item Suppose that $d_{ii} = 0$ for all $i \in \{2, \ldots, n-2\}$ and there exist $k, \ell \in \{2, \ldots, n-2\}$ such that $k < \ell$, $d_{k{\ell}} \neq 0$. Let $d_{23} = 1$ without loss of generality. Then
\[\begin{split}in_{\omega^{(5)}}(I') = &(y_4, \ldots, y_{n-2})S_1 +\\&(t^2, y_2t, y_3t, xy_1, xy_2, xy_3, y_1^2-xt, y_1y_2, y_1y_3, y_2^2, y_2y_3-y_1t, y_3^2).\end{split}\]

and, since $S/in_{\omega^{(5)}}(I')$ is Koszul (being a trivial fiber extension of the Koszul algebra (6)), $S/I'$ is too.

\item Suppose $d_{k{\ell}} = 0$ for all $k, \ell \in \{2, \ldots, n-2\}$ such that $k \leq \ell$ and suppose there exists $i \in \{2, \ldots, n-2\}$ such that $d_{1i} \neq 0$. Let $d_{12} = 1$ without loss of generality. Then
\[in_{\omega^{(6)}}(I') = (y_3, \ldots, y_{n-2})S_1 + (t^2, y_2t, xy_1, xy_2, y_1^2-xt, y_1t, y_2^2)\]
and hence, since $S/in_{\omega^{(6)}}(I')$ is a trivial fiber extension of the Koszul algebra (7), $S/I'$ is also Koszul.

\item Suppose $d_{k{\ell}} = 0$ for all $k \in \{1, \ldots, n-2\}$, $\ell \in \{2, \ldots, n-2\}$ such that $k \leq \ell$. Then $d_{0i} \neq 0$ for all $i \in \{2, \ldots, n-2\}$. Without loss of generality, let $d_{02} = 1$. One has that
\[in_{\omega^{(7)}}(I') = (y_3, \ldots, y_{n-2})S_1 + (t^2, y_2t, xy_1, xy_2-y_1t, y_1^2-xt, y_1y_2, y_2^2).\]
and hence, since $S/in_{\omega^{(7)}}(I')$ is a trivial fiber extension of the Koszul algebra (8), $S/I'$ is Koszul as well.

\end{itemize}

This ends the discussion of the case when $\mathcal{V}(I)$ consists of a single point and $R$ contains a null-square linear form of rank 2. \qed

\subsection{The variety $\mathcal{V}(I)$ consists of a point and all null-square linear forms in $R$ have rank 1} \label{onepointrank1}

\begin{table}[h]
\centering
\caption{Integral weights used in Section \ref{onepointrank1}}
\label{weighttable2}
    \begin{tabular}{|c|c|c|c|c|}
    \hline
     & $x$ & $y$ & $t_1$ & $t_2$ \\ \hline
    $\omega^{(8)}$ & 3 & 2    & 2    & 1                                               \\
    $\omega^{(9)}$ & 8 & 6    & 4    & 1                                           \\
    $\omega^{(10)}$ & 5 & 4    & 3    & 1                                          \\ \hline
    \end{tabular}
\end{table}

Let $S = \mathbbm{k}[x, y, t_1, \ldots, t_{n-2}]$. We have proved in Section \ref{onepointrank2} that, when $\mathcal{V}(I)$ is a single point, the presence of a null-square linear form of rank 2 guarantees Koszulness. The cases left to analyse are then the ones where all null-square linear forms have rank 1. Recall that, by Lemma \ref{linear_forms_and_points}, null-square linear forms correspond to points of $\mathcal{V}(V^{\perp})$. We know that $x^2 \in V^{\perp}$, hence we can complete $x^2$ to a $\mathbbm{k}$-basis of $V^{\perp}$ \[\{G_1 := x^2, G_2, G_3\}.\] Solving the system $G_1 = G_2 = G_3 = 0$ is equivalent to solving the system $G_2|_{x=0} = G_3|_{x=0} = 0$. Note that in this process of reduction we lose both a variable and an equation, hence coming back to the case of codimension 2 (in $n-1$ variables) studied by Conca in \cite{Conca2000}. In particular, going through the proof of \cite[Proposition 6]{Conca2000}, we get that, in this restricted setting, a necessary condition for all null-square linear forms to have rank 1 is that one of $G_2|_{x=0}$ and $G_3|_{x=0}$ is a square of a linear form (see Lemma \ref{concalemma} at the end of this section). Without loss of generality, let us assume that $G_2|_{x=0} = y^2$. This means that \[G_2 = y^2+xL,\] where $L$ is a nonzero linear form not containing $x$ nor $y$ (if it contained $y$, we could perform a suitable change of coordinates to eliminate it). For the rest of this subsection, we shall use the following as a $\mathbbm{k}$-basis of $V^{\perp}$: \[\begin{split}G_1 &= x^2 \\ G_2 &= y^2+xL(t_1, \ldots, t_{n-2}) \\ G_3 &= 2axy + xL_2(t_1, \ldots, t_{n-2}) + yL_3(t_1, \ldots, t_{n-2}) + Q(t_1, \ldots, t_{n-2}),\end{split}\] where $L$, $L_2$ and $L_3$ are linear forms, $Q$ is a quadratic form and $a \in \mathbbm{k}$. We are going to analyse several subcases according to the rank of the quadric $Q$ (call this number $r$).

\subsubsection{The rank of Q is greater than or equal to 2}
In this case, after a suitable change of coordinates, we may assume that $Q = t_1t_2 + t_3^2 + \ldots + t_r^2$. Now, if $L$ contains $t_i$ where $i \in \{1, 2\}$, then $t_i$ is a null-square linear form of rank 2. Let us see this computation in detail, since we will hint at analogous ones several times in the next lines. Note first that $\overline{t_i^2}$ is zero by applying Lemma \ref{linear_forms_and_points}. As regards the rank of $t_i$, note that we are in the hypotheses of Remark \ref{apolarityrem} with $m_1 = x^2$, $m_2 = y^2$ and $m_3 = t_1t_2$. Hence, to compute the dimension of $\overline{t_i}R_1$ we just have to compute the rank of the matrix that encodes the multiplication by $\overline{t_i}$ with respect to the basis $\{\overline{x^2}, \overline{y^2}, \overline{t_1t_2}\}$, i.e. (assuming without loss of generality that $i=1$)

\[
  \bordermatrix
  {
    ~ & \overline{xt_1} & \overline{yt_1} & \overline{t_1^2} & \overline{t_1t_2} & \overline{t_1t_3} & \ldots & \overline{t_1t_{n-2}} \cr
    \overline{x^2} & 0 & 0 & 0 & 0 & 0 & \ldots & 0 \cr
    \overline{y^2} & \neq 0 & 0 & 0 & 0 & 0 & \ldots & 0 \cr
    \overline{t_1t_2} & * & * & 0 & 1 & 0 & \ldots & 0 \cr
   }
\]
The rank of this matrix is clearly 2 and thus we are done.

If $L$ does not contain neither $t_1$ nor $t_2$, then it must contain $t_j$, where $j \in \{3, \ldots, n-2\}$: if $j \leq r$, then $t_1 - t_2 + t_j$ is a null-square linear form of rank 2, whereas if $j > r$ then $t_1 + t_j$ is such a form.

\subsubsection{The rank of Q is 1}
We can use our remaining degrees of freedom to assume that $Q = t_1^2$.
If the coefficient of $t_i$ in $L$ is nonzero for some $i \neq 1$ then, possibly after a change of coordinates, we can further assume that $L = 2t_2$; otherwise, by rescaling $y$ in a suitable way, we can assume that $L = 2t_1$. In both cases, acting on $t_1$ we can also assume that the coefficient of $xt_1$ in $G_3$ is zero, i.e. $L_2$ does not contain $t_1$.

\begin{itemize}
\item $L = 2t_2$. If $L_3$ contains $t_i$ with $i \neq 1$, then we get a null-square linear form of rank 2. More precisely, if $L_3$ contains $t_2$, then $t_2$ is such a form; otherwise, $t_2 + t_j$ is such a form for any $t_j$ (with $j \in \{3, \ldots, n-2\}$) appearing in $L_3$. Let us then assume that $L_3 = 2bt_1$. By sending $t_1$ into $t_1-by$ we may assume that $L_3$ is, in fact, zero. If we have more then five variables, then the algebra we are studying is a trivial fiber extension. 

If we have exactly five variables, then $L_2$ must contain $t_3$ and after a suitable change of coordinates we get $a = 0$ and $L_2 = 2t_3$. Then \begin{equation*}I = (xy, xt_1, xt_2-y^2, xt_3-t_1^2, yt_1, yt_2, yt_3, t_1t_2, t_1t_3, t_2^2, t_2t_3, t_3^2)\end{equation*} and this gives us the Koszul algebra (9) in Table \ref{koszultable}.

If we have exactly four variables, then (recalling that $L_2$ does not contain $t_1$) \[G_3 = 2axy + 2cxt_2 + t_1^2\] and \[I = (xy - at_1^2, xt_1, xt_2 - y^2 - ct_1^2, yt_1, yt_2, t_1t_2, t_2^2).\]
Note that at least one of $a$ and $c$ must be nonzero, otherwise $\mathcal{V}(I)$ would contain two points against our assumptions.

If $c \neq 0$, then by rescaling $t_1$ we may assume $c = 1$. Using the weight $\omega^{(8)}$ (see Table \ref{weighttable2}), one gets that 
\[in_{\omega^{(8)}}(I) = (xy, xt_1, xt_2-y^2-t_1^2, yt_1, yt_2, t_1t_2, t_2^2)\]
and, since $S/in_{\omega^{(8)}}(I)$ is the Koszul algebra (10) in Table \ref{koszultable}, $S/I$ is Koszul too.

If $c = 0$, then we may assume after rescaling $t_1$ that $a = 1$. Then $S/I$ is the Koszul algebra (11) in Table \ref{koszultable}.

\item $L = 2t_1$. We cannot have more than five variables, otherwise we would get a trivial fiber extension. If the variables are exactly five and our algebra is not a trivial fiber extension, then after a suitable change of coordinates we get $a = 0$, $L_2 = 2t_2$, $L_3 = 2t_3$. Then \[I = (xy, xt_1-y^2, xt_2-t_1^2, xt_3, yt_1, yt_2, yt_3-t_1^2, t_1t_2, t_1t_3, t_2^2, t_2t_3, t_3^2)\] and this gives us the Koszul algebra (12) in Table \ref{koszultable}. What if we have exactly four variables? Then the variable $t_2$ must appear in at least one of $L_2$ and $L_3$.

If $t_2$ appears in $L_3$, then $G_3 = 2axy + 2dxt_2 + 2eyt_1 + 2yt_2 + t_1^2$ (rescaling $t_2$ so that the coefficient of $yt_2$ is 2). Mapping simultaneously $t_1$ into $t_1 + dex$ and $t_2$ into $t_2 - et_1 + (-de^2 - a)x$ we find the new basis \[\{G'_1 = G_1,~G'_2 = G_2,~G'_3 = 2dxt_2 + 2yt_2 + t_1^2\}\] and hence \[I' = (xy, xt_1-y^2, xt_2-dt_1^2, yt_1, yt_2-t_1^2, t_1t_2, t_2^2).\]
One has that
\[in_{\omega^{(9)}}(I') = (xy, xt_1-y^2, xt_2, yt_1, t_1^2, t_1t_2, t_2^2).\]
We find out that, up to variable renaming, $S/in_{\omega^{(9)}}(I')$ is the Koszul algebra (7) in Table \ref{koszultable}. Hence $S/I$ is Koszul as well.

If $t_2$ does not appear in $L_3$, then it appears in $L_2$ and rescaling it suitably we get that $G_3 = 2axy + 2xt_2 + *yt_1 + t_1^2$, hence \[I = (xy-at_1^2, xt_1-y^2, xt_2-t_1^2, yt_1-*t_1^2, yt_2, t_1t_2, t_2^2).\] One has that \[in_{\omega^{(10)}}(I) = (xy, xt_1-y^2, xt_2-t_1^2, yt_1, yt_2, t_1t_2, t_2^2).\] Since $S/in_{\omega^{(10)}}(I)$ is the Koszul algebra (11) with the roles of $y$ and $t_1$ switched, $S/I$ is Koszul.
\end{itemize}

\subsubsection{The quadric Q equals zero}
In this case, after a change of coordinates, we may assume $L = 2t_1$ without loss of generality. If $L_3$ is nonzero, then we get a null-square linear form of rank 2. More precisely, if $t_1$ appears in $L_3$, then $t_1$ is such a form; otherwise, take $t_1 + t_j$ where $t_j$ appears in $L_3$.

Hence $L_3=0$ and we must have exactly four variables, otherwise our algebra would be a trivial fiber extension. After a change of coordinates we can assume that $a=0$ and $L_2 = t_2$. This means that \[I = (xy, y^2-xt_1, yt_1, yt_2, t_1^2, t_1t_2, t_2^2)\] and this gives us the Koszul algebra (13) in Table \ref{koszultable}.

This ends the proof of Theorem \ref{mainthm}. \qed

\begin{lemma}[essentially \cite{Conca2000}] \label{concalemma}
Let $\mathbbm{k}$ be an algebraically closed field of characteristic different from 2 and let $R$ be a quadratic standard graded $\mathbbm{k}$-algebra which is not a trivial fiber extension. If $\dim_{\mathbbm{k}}R_1 \geq 3$, $\dim_{\mathbbm{k}}R_2 = 2$ and there is no null-square linear form $\ell \in R_1$ such that $\ell R_1 = R_2$, then $V^{\perp}$ contains the square of a linear form.
\end{lemma}
\begin{proof}
We shall refer to the division into three cases that appears in the proof of \cite[Proposition 6]{Conca2000}.

In Case 1, Conca proves that a necessary condition in order not to have any $\ell$ such that $\ell R_1 = R_2$ is that the defining ideal $I$ is contained in the ideal $(x, z-fy, t_i-b_iy \mid i \in \{3, \ldots, n-1\})$. Sending $z$ into $z+fy$ and each $t_i$ into $t_i+b_iy$ one then gets that $y^2$ does not appear in any generator of the new ideal $I'$. Hence $y^2$ lies in the orthogonal of the associated vector space $V'$ and, since going back to $V^{\perp}$ just involves a change of coordinates, we are done. In Case 2 there always exists a null-square linear form $\ell$ such that $\ell R_1 = R_2$. Finally, Case 3 is the same as Case 1 but for a single three-variable case where $I = (x^2, xy, yz, y^2-xz)$ and hence $V^{\perp}$ contains $z^2$. 
\end{proof}

\section{Some Koszul algebras which are not G-quadratic}
For the rest of the section, let $S = \mathbbm{k}[x, y, z, t]$.

\begin{prop} \label{Gquadneccond}
Let $I$ be a quadratic ideal of $S$ such that the Hilbert series of $S/I$ is $\mathbf{A}_4$. If $S/I$ is G-quadratic, then there exists a change of coordinates $g \in \GL_4(\mathbbm{k})$ such that $g(I) \in \mathfrak{F}$, where $\mathfrak{F}$ is the family consisting of the quadratic ideals of the form \[(y^2, zt, xy, xz-*yt, t^2-*yt-*yz, z^2-*yt-*yz, xt-*yt-*yz).\]
\end{prop}
\begin{proof}
See Appendix \ref{appendixGquad}.
\end{proof}

Let us examine some consequences of Proposition \ref{Gquadneccond}.

\begin{rem} \label{atleasttwoforms}
If $S/I$ has Hilbert series $\mathbf{A}_4$ and is G-quadratic, then it has at least two distinct null-square linear forms. One of them is trivially $y$. To find another one, recall that, by Lemma \ref{linear_forms_and_points}, the null-square linear forms of $S/I$ correspond to the points of the projective variety $\mathcal{V}(V^{\perp})$. In our case, 
\[I = (y^2, zt, xy, xz-*yt, t^2+2ayz+2byt, z^2+2cyz+2dyt, xt-*yt-*yz)\]
and 
\[V^{\perp} = \vspan\{x^2,~yz-at^2-cz^2-*xt,~yt-bt^2-dz^2-*xt-*xz\}.\]
If $a=b=0$ one has that, since $t^2$ does not appear in $V^{\perp}$, $t$ is a null-square linear form of $S/I$; otherwise, $(a\lambda^2 + c)y + z + \lambda t$ (where $\lambda$ is a root of $aX^3-bX^2+cX-d$) is a null-square linear form.
\end{rem}

The following corollary settles in the negative the following question asked by Conca in \cite[Section 4]{Conca2009}: if $R$ is a standard graded quadratic $\mathbbm{k}$-algebra which is not a trivial fiber extension and is such that $\dim_{\mathbbm{k}}R_1 > 3 = \dim_{\mathbbm{k}}R_2$, is it true that $R$ is G-quadratic?

\begin{cor} \label{notGquad}
The algebras (5), (10), and (11) in Table \ref{koszultable} are Koszul but not G-quadratic.
\end{cor}
\begin{proof}
The Koszulness of the algebras (5), (10), (11) was already established in Lemma \ref{allkoszul}.
Suppose now that any of those were G-quadratic. By Remark \ref{atleasttwoforms}, then, we should find at least two null-square linear forms, but it is immediate to check that there exists only one. Contradiction.
\end{proof}

\begin{rem}
Via a subtler analysis of the family $\mathfrak{F}$ in Proposition \ref{Gquadneccond}, we are also able to prove that the Koszul algebras (4), (7), and (8) in Table \ref{koszultable} are not G-quadratic. The argument still relies on the properties of null-square linear forms for arbitrary elements of $\mathfrak{F}$.
\end{rem}

\newpage
\appendix
\section{Proof of Lemma \ref{allkoszul}} \label{appendixkoszul}
Our goal is to show that every algebra in Table \ref{koszultable} admits a Koszul filtration that works in all characteristics. We first need to build some tools.

Let $S$ be the polynomial ring $\mathbbm{k}[x_1, \ldots, x_n]$ and write $\mathbf{x}^{\mathbf{u}}$ as a shorthand to denote the monomial $x_1^{u_1}x_2^{u_2}\ldots x_n^{u_n}$ in $S$. Following Kahle and Miller \cite[Definition 2.15]{KahleMiller}, we say that an ideal $I$ in $S$ is a \emph{unital binomial ideal} if it can be generated by a set of monomials and pure-difference binomials, i.e. binomials of the form $\mathbf{x}^{\mathbf{u}} - \mathbf{x}^{\mathbf{v}}$.
Refining very slightly Proposition 1.1 and Corollary 1.7 in \cite{EisenbudSturmfels}, we get the following result:

\begin{lemma} \label{unitalbinomial}
Let $I$ be a unital binomial ideal in $S = \mathbbm{k}[x_1, \ldots, x_n]$ and fix a monomial order $\tau$ on $S$.
\begin{enumerate}
\item The reduced Gr\"obner basis $\mathcal{G}$ of $I$ with respect to $\tau$ is made of monomials and pure-difference binomials and does not depend on the characteristic of the field $\mathbbm{k}$. As a consequence, the Hilbert series of $S/I$ is characteristic-free.
\item The normal form of any monomial with respect to $\tau$ modulo $\mathcal{G}$ is either a monomial or 0. This computation does not depend on the characteristic of the field $\mathbbm{k}$. 
\item Any set of monomials and pure-difference binomials which generates $I$ in some characteristic is a generating set for $I$ in all characteristics.
\item Let $\mathbf{x}^{\mathbf{a}}$ be a monomial in $S$. Then $I :_S \mathbf{x}^{\mathbf{a}}$ is a unital binomial ideal.
\end{enumerate}
\end{lemma}
\begin{proof}
Parts (i) and (ii) are immediate applications of the Buchberger algorithm (with a final interreduction) and of the division algorithm, respectively.

Let us prove part (iii).
Let $U = \{\mathbf{x}^{\mathbf{u_1}}, \ldots, \mathbf{x}^{\mathbf{u_r}}, \mathbf{x}^{\mathbf{v_1}}-\mathbf{x}^{\mathbf{v'_1}}, \ldots, \mathbf{x}^{\mathbf{v_s}} - \mathbf{x}^{\mathbf{v'_s}}\}$ be a generating set for $I$ in characteristic $c$. Then, by part (ii), one has that $(U) \subseteq I$ in all characteristics and, since the Hilbert series of both sides do not depend on the characteristic by part (i), equality holds in all characteristics.

As regards part (iv), we claim that
\[
U := \{\mathbf{x}^{\mathbf{u}} \mid \mathbf{x}^{\mathbf{u}+\mathbf{a}} \in I\} \cup \{\mathbf{x}^{\mathbf{u}} - \mathbf{x}^{\mathbf{v}} \mid \NF_{\mathcal{G}}(\mathbf{x}^{\mathbf{u}+\mathbf{a}}) = \NF_{\mathcal{G}}(\mathbf{x}^{\mathbf{v}+\mathbf{a}}) \neq 0\}
\]
is a generating set for $I :_S \mathbf{x}^{\mathbf{a}}$.

It is easy to see that all elements in $U$ belong to $I :_S \mathbf{x}^{\mathbf{a}}$. On the other hand, let $f \in I :_S \mathbf{x}^{\mathbf{a}}$ and assume without loss of generality that no $\mathbf{x}^{\mathbf{u}}$ in the support of $f$ is such that $\mathbf{x}^{\mathbf{u}+\mathbf{a}}$ lies in $I$. Then, by part (ii), the normal forms modulo $\mathcal{G}$ of all monomials in the support of $f \cdot \mathbf{x}^{\mathbf{a}}$ form a set of distinct nonzero monomials $\{\mathbf{x}^{\mathbf{b_1}}, \ldots, \mathbf{x}^{\mathbf{b_k}}\}$. Now write $f$ as $f_1 + \ldots + f_k$, where the support of each $f_i$ contains precisely the monomials $\mathbf{x}^{\mathbf{u}}$ of $f$ such that $\NF_{\mathcal{G}}(\mathbf{x}^{\mathbf{u}+\mathbf{a}}) = \mathbf{x}^{\mathbf{b_i}}$. If we prove that each $f_i$ can be written as a combination of elements in $U$, we are done. Let then $f_1 = \sum_{i=1}^{N}c_i \cdot \mathbf{x}^{\mathbf{u_i}}$, where $c_i \in \mathbbm{k}$ for all $i$. By construction one must have that $\sum_{i=1}^{N}c_i = 0$ and hence, since \[f_1 = c_1(\mathbf{x}^{\mathbf{u_1}}-\mathbf{x}^{\mathbf{u_2}}) + (c_1+c_2)(\mathbf{x}^{\mathbf{u_2}}-\mathbf{x}^{\mathbf{u_3}}) + \ldots + (\sum_{i=1}^{N-1}c_i)(\mathbf{x}^{\mathbf{u_{N-1}}} - \mathbf{x}^{\mathbf{u_N}}),\] the claim follows.
\end{proof}

\begin{cor} \label{monomialfiltrations}
If $I$ is a unital binomial ideal and $\mathfrak{F}$ is a monomial Koszul filtration (i.e. a filtration where all ideals are generated by variables) for $I$ in some characteristic, then $\mathfrak{F}$ is a Koszul filtration for $I$ in all characteristics. 
\end{cor}
\begin{proof}
By hypothesis, there exists a characteristic $c$ where $\mathfrak{F}$ is a monomial Koszul filtration. Consider $I$ to be inside $S := \mathbbm{k}[x_1, \ldots, x_n]$, where $\mathbbm{k}$ has characteristic $c$.
For each $(\overline{x}_{i_1}, \ldots, \overline{x}_{i_m}) \in \mathfrak{F}$ one has (after some renumbering of the indices) that $(\overline{x}_{i_2}, \ldots, \overline{x}_{i_m}) \in \mathfrak{F}$ and $(\overline{x}_{i_2}, \ldots, \overline{x}_{i_m}) \ {:}_{S/I} \ (\overline{x}_{i_1}) \in \mathfrak{F}$. The last statement means that, at least in characteristic $c$, 
\begin{equation} \label{colon_formula} \tag{$\diamondsuit$}
(I + (x_{i_2}, \ldots, x_{i_m})) :_S (x_{i_1}) = I + (x_{j_1}, \ldots, x_{j_r}),
\end{equation}
where $(\overline{x}_{j_1}, \ldots, \overline{x}_{j_r}) \in \mathfrak{F}$.

Now the left hand side of \eqref{colon_formula} is unital by Lemma \ref{unitalbinomial}.(iv), whereas the right hand side is generated by monomials and pure-difference binomials. Hence, by Lemma \ref{unitalbinomial}.(iii), \eqref{colon_formula} must hold in all characteristics.
\end{proof}

\begin{proof}[Proof of Lemma \ref{allkoszul}]
Consider Table \ref{koszulfiltrations}. The monomial Koszul filtrations exhibited there for the algebras $(1)$--$(9)$ and $(12)$--$(13)$ (defined by unital binomial ideals) work in characteristic $0$, as it can be checked for instance by using \cocoa. Hence, by Corollary \ref{monomialfiltrations}, these are Koszul filtrations in all characteristics. We are left with two cases: the algebras $(10)$ and $(11)$.
If $\fchar\mathbbm{k} = 2$, one checks that the proposed filtrations for $(10)$ and $(11)$ work. For the rest of the proof, let $\fchar\mathbbm{k} \neq 2$ and let $I$ and $J$ be the ideals defining the algebras (10) and (11) respectively. 

Even for these two objects, most of the colon computations we need fit the setting of Corollary \ref{monomialfiltrations} (possibly after a suitable change of coordinates). The only computations that have to be checked by hand are the following: 
\[\begin{array}{c c}
0 :_{S/I} (\overline{x}) = (\overline{y}, \overline{t}_1), & 0 :_{S/I} (\overline{y}) = (\overline{x}, \overline{t}_1, \overline{t}_2), \\
0 :_{S/J} (\overline{x + t}_1) = (\overline{y - t}_1), & 0 :_{S/J} (\overline{y - t}_1) = (\overline{x + t}_1, \overline{t}_2).
\end{array}\]

$J$ is unital binomial, whereas $I$ can become so after a suitable change of coordinates, since $\fchar\mathbbm{k} \neq 2$. As a consequence, by Lemma \ref{unitalbinomial}.(i) the Hilbert series of $I$, $J$, $I+(x)$, $I+(y)$ and $J+(y-t_1)$ are characteristic-free. Since $J+(x+t_1)$ becomes monomial after sending $x$ into $x-t_1$, its Hilbert series is also characteristic-free. Hence, the Hilbert series of $I :_S (x)$, $I :_S (y)$, $J :_S (x+t_1)$, $J :_S (y-t_1)$ are characteristic-free too.
One checks that, in any characteristic,
\[\begin{array}{c c}
(y, t_1, t_2^2, xt_2) \subseteq I :_S (x), & (x, y^2, t_1, t_2) \subseteq I :_S (y), \\
(y-t_1, xt_1, xt_2, t_1^2, t_1t_2, t_2^2) \subseteq J :_S (x+t_1), & (x+t_1, t_2, t_1^2, y^2, yt_1) \subseteq J :_S (y-t_1).
\end{array}\]
and Hilbert series are the expected ones.
Finally, one checks by hand that going modulo $I$ or $J$ respectively yields the desired results regardless of the characteristic.

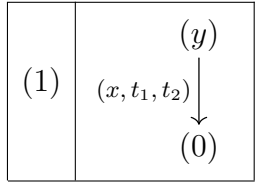
\begin{table}[h]
\caption[]{
Every arrow in the diagrams below represents the computation of a colon ideal. For instance,

\vspace{5pt}
\begin{tabulary}{\textwidth}{c | c | c | c}

\cline{2-3}
\hspace{0.3\textwidth}
&
(1)
&
\begin{tikzcd}[column sep=small]
(y) \arrow{d}[swap]{(x, t_1, t_2)}\\
(0)
\end{tikzcd}
&
\hspace{0.3\textwidth}
\\
\cline{2-3}

\end{tabulary}

\vspace{5pt}
should be read as ``$0 :_{S/I} (\overline{y}) = (\overline{x}, \overline{t}_1, \overline{t}_2)$'', where $I$ is the ideal corresponding to line $(1)$ in Table \ref{koszultable}.
}
\label{koszulfiltrations}

\begin{adjustbox}{center, max width = \paperwidth}
\begin{tabular}[t]{| c | c | c | c |}
\hline
(1) &
\begin{tikzcd}[column sep=small]
\mathfrak{m} = (x, y, t_1, t_2, t_3) \arrow{d}[swap]{(y, t_1, t_2, t_3)} \\
(y, t_1, t_2, t_3) \arrow{d}[swap]{\mathfrak{m}}\\
(y, t_1, t_3) \arrow{d}[swap]{\mathfrak{m}} & (x, t_1, t_2) \arrow{d}{\mathfrak{m}}\\
(y, t_1) \arrow{d}[swap]{\mathfrak{m}} & (x, t_2) \arrow{d}{\mathfrak{m}}\\
(y) \arrow{d}[swap]{(x, t_1, t_2)} & (x) \arrow{dl}{(y, t_1, t_3)}\\
(0)
\end{tikzcd}
&
\begin{tikzcd}[column sep=small]
& \mathfrak{m} = (x, y, t_1, t_2) \arrow{d}{\mathfrak{m}} \\
(x, t_1, t_2) \arrow{dr}[swap]{\mathfrak{m}} & (x, y, t_2) \arrow{d}{(x, t_1, t_2)}\\
& (x, t_2) \arrow{d}[swap]{\mathfrak{m}} & (y, t_1) \arrow{d}{(x, y, t_2)}\\
& (x) \arrow{d}[swap]{(y, t_1)} & (y) \arrow{dl}{(x, t_1, t_2)}\\
& (0)
\end{tikzcd}
&
(2)
\\
\hline
(3)
&
\begin{tikzcd}[column sep=small]
\mathfrak{m} = (x, y, t_1, t_2) \arrow{d}[swap]{(x, t_1, t_2)} \\
(x, t_1, t_2) \arrow{d}[swap]{\mathfrak{m}}\\
(x, t_1) \arrow{d}[swap]{\mathfrak{m}} & (y, t_2) \arrow{d}{\mathfrak{m}}\\
(x) \arrow{d}[swap]{(y, t_2)} & (y) \arrow{dl}{(x, t_1)}\\
(0)
\end{tikzcd}
&
\begin{tikzcd}[column sep=small]
\mathfrak{m} = (x, y_1, y_2, t) \arrow{d}[swap]{\mathfrak{m}} \\
(x, y_1, t) \arrow{d}[swap]{\mathfrak{m}} & (x, y_1, y_2) \arrow{dl}{\mathfrak{m}}\\
(x, y_1) \arrow{d}[swap]{(y_1, y_2)} & (y_1, y_2) \arrow{dl}{(x, y_1, t)}\\
(y_1) \arrow{d}[swap]{(x, y_1, y_2)}\\
(0)
\end{tikzcd}
&
(4)
\\
\hline
\end{tabular}
\end{adjustbox}
\end{table}

\begin{table}[h]
\begin{adjustbox}{center, max width = \paperwidth}
\begin{tabular}[t]{| c | c | c | c |}
\hline
(5)
&
\begin{tikzcd}[column sep=small]
\mathfrak{m} = (x, y_1, y_2, t) \arrow{d}[swap]{\mathfrak{m}} \\
(x, y_1, t) \arrow{d}[swap]{\mathfrak{m}} & (x, y_2, t) \arrow{dl}{\mathfrak{m}}\\
(x, t) \arrow{d}[swap]{(x, y_2, t)} & (y_1, y_2) \arrow{d}{(x, y_2, t)}\\
(x) \arrow{d}[swap]{(y_1, y_2)} & (y_2) \arrow{dl}{(x, y_1, t)}\\
(0)
\end{tikzcd}
&
\begin{tikzcd}[column sep=small]
& \mathfrak{m} = (x, y_1, y_2, y_3, t) \arrow{d}[swap]{\mathfrak{m}} \\
& (x, y_1, y_2, t) \arrow{d}[swap]{\mathfrak{m}}\\
(x, y_2, y_3) \arrow{dr}[swap]{\mathfrak{m}} & (x, y_1, y_2) \arrow{d}[swap]{\mathfrak{m}} & (y_1, y_2, y_3) \arrow{d}{\mathfrak{m}}\\
& (x, y_2) \arrow{d}[swap]{(x, y_1, y_2, t)} & (y_1, y_3) \arrow{d}{\mathfrak{m}}\\
& (x) \arrow{d}[swap]{(y_1, y_2, y_3)} & (y_1) \arrow{dl}{(x, y_2, y_3)}\\
& (0)
\end{tikzcd}
&
(6)
\\
\hline
(7)
&
\begin{tikzcd}[column sep=small]
\mathfrak{m} = (x, y_1, y_2, t) \arrow{d}[swap]{\mathfrak{m}} \\
(x, y_2, t) \arrow{d}[swap]{(x, y_2, t)}\\
(x, t) \arrow{d}[swap]{\mathfrak{m}} & (y_1, y_2) \arrow{d}{(x, y_2, t)}\\
(x) \arrow{d}[swap]{(y_1, y_2)} & (y_2) \arrow{dl}{(x, y_2, t)}\\
(0)
\end{tikzcd}
&
\begin{tikzcd}[column sep=small]
\mathfrak{m} = (x, y_1, y_2, t) \arrow{d}[swap]{\mathfrak{m}} \\
(x, y_2, t) \arrow{d}[swap]{\mathfrak{m}}\\
(x, y_2) \arrow{d}[swap]{\mathfrak{m}}\\
(x) \arrow{d}[swap]{(y_1)} & (y_1) \arrow{dl}{(x, y_2)}\\
(0)
\end{tikzcd}
&
(8)
\\
\hline
(9)
&
\begin{tikzcd}[column sep=small]
\mathfrak{m} = (x, y, t_1, t_2, t_3) \arrow{d}[swap]{\mathfrak{m}} \\
(x, y, t_2, t_3) \arrow{d}[swap]{\mathfrak{m}} & (x, t_1, t_2, t_3) \arrow{dl}{\mathfrak{m}}\\
(x, t_2, t_3) \arrow{d}[swap]{\mathfrak{m}}\\
(x, t_3) \arrow{d}[swap]{\mathfrak{m}} & (y, t_1) \arrow{d}{(x, y, t_2, t_3)}\\
(x) \arrow{d}[swap]{(y, t_1)} & (y) \arrow{dl}{(x, t_1, t_2, t_3)}\\
(0)
\end{tikzcd}
&
\begin{tikzcd}[column sep=small]
& \mathfrak{m} = (x, y, t_1, t_2) \arrow{d}{\mathfrak{m}} \\
(x, t_1, t_2) \arrow{dr}[swap]{(x, y, t_2)} & (x, y, t_2) \arrow{d}{(x, t_1, t_2)}\\
& (x, t_2) \arrow{d}[swap]{\mathfrak{m}} & (y, t_1) \arrow{d}{(x, y, t_2)}\\
& (x) \arrow[swap]{d}{(y, t_1)} & (y) \arrow{dl}{(x, t_1, t_2)}\\
& (0)
\end{tikzcd}
&
(10)
\\
\hline
\end{tabular}
\end{adjustbox}
\end{table}

\begin{table}[h]
\begin{adjustbox}{center, max width = \paperwidth}
\begin{tabular}[t]{| c | c | c | c |}
\hline
(11)
&
\begin{tikzcd}[column sep=small]
\mathfrak{m} = (x, y, t_1, t_2) \arrow{d}[swap]{\mathfrak{m}} \\
(x, t_1, t_2) \arrow{d}[swap]{\mathfrak{m}}\\
(x+t_1, t_2) \arrow{d}[swap]{\mathfrak{m}}\\
(x+t_1) \arrow[swap]{d}{(y-t_1)} & (y-t_1) \arrow{dl}{(x+t_1, t_2)}\\
(0)
\end{tikzcd}
&
\begin{tikzcd}[column sep=small]
\mathfrak{m} = (x, y, t_1, t_2, t_3) \arrow{d}[swap]{\mathfrak{m}} \\
(x, t_1, t_2, t_3) \arrow{d}[swap]{\mathfrak{m}}\\
(x, t_2, t_3) \arrow{d}[swap]{(y, t_2, t_3)} & (y, t_2, t_3) \arrow{dl}{(x, t_1, t_2, t_3)}\\
(t_2, t_3) \arrow{d}[swap]{\mathfrak{m}}\\
(t_3) \arrow{d}[swap]{(x, t_1, t_2, t_3)}\\
(0)
\end{tikzcd}
&
(12)
\\
\hline
(13)
&
\begin{tikzcd}[column sep=small]
\mathfrak{m} = (x, y, t_1, t_2) \arrow{d}[swap]{\mathfrak{m}} \\
(x, t_1, t_2) \arrow{d}[swap]{\mathfrak{m}}\\
(x, t_2) \arrow{d}[swap]{\mathfrak{m}}\\
(x) \arrow{d}[swap]{(y)} & (y) \arrow{dl}{(x, t_1, t_2)}\\
(0)
\end{tikzcd}
\\
\cline{1-2}
\end{tabular}
\end{adjustbox}
\end{table}
\end{proof}

\clearpage
\section{Proof of Proposition \ref{Gquadneccond}} \label{appendixGquad}
For this Appendix, let $S = \mathbbm{k}[x, y, z, t]$.

To prove Proposition \ref{Gquadneccond} we first need the following easy lemma.

\begin{lemma} \label{uniquemonomial}
The only (up to variable renaming) monomial quadratic ideal $J$ in $S$ such that the Hilbert series of $S/J$ is $\mathbf{A}_4$ is $(y^2, z^2, t^2, xy, xz, xt, zt)$.
\end{lemma}
\begin{proof}
Let $V$ be the vector subspace of $S_2$ generated by the monomials minimally generating $J$. Since the Hilbert polynomial of $S/J$ is 1, $\mathcal{V}(J)$ consists of exactly one point. Using the apolarity pairing and in particular Lemma \ref{linear_forms_and_points}, we can equivalently require that $V^{\perp}$ contains the square of only one linear form, i.e. variable: let it be $x^2$ without loss of generality. Note that, if $V^{\perp}$ contained $xy$ (or $xz$ or $xt$), then the Hilbert series of $S/J$ would not be $\mathbf{A}_4$, as $\overline{x}^3$ and $\overline{x^2y}$ (or $\overline{x^2z}$ or $\overline{x^2t}$) would be linearly independent elements of $(S/J)_3$. Then the remaining two generators of $V^{\perp}$ can be chosen only between $yz$, $yt$ and $zt$: it is easy to see that all possible choices are equivalent and yield $J = (y^2, z^2, t^2, xy, xz, xt, zt)$ up to variable renaming.
\end{proof}

\begin{proof}[Proof of Proposition \ref{Gquadneccond}]
By definition, since $R$ is G-quadratic, there exist a change of coordinates $h \in \GL_4(\mathbbm{k})$ and a term order $\tau$ on $S$ such that the generators of $h(I)$ (let them be $f_1$, \ldots, $f_7$) 
 form a Gr\"obner basis with respect to $\tau$. By Lemma \ref{uniquemonomial}, we can assume (possibly after some variable renaming) that \[in_{\tau}(h(I)) = (y^2, z^2, t^2, xy, xz, xt, zt) = (in_{\tau}(f_1), \ldots, in_{\tau}(f_7)).\] Our goal is to show that, no matter what $\tau$ is, we can always change coordinates so that we land on an element of the family $\mathfrak{F}$. To do this, we divide possible term orders on $S$ with respect to the way they order the variables. Note that, since $z$ and $t$ play an identical role in $(y^2, z^2, t^2, xy, xz, xt, zt)$, we can assume without loss of generality that $z >_{\tau} t$. We need to examine twelve cases.
\renewcommand{\labelenumi}{\arabic{enumi}.}
\begin{enumerate}

\item \label{case 3} $x > z > t > y$: \[\begin{split}h(I) = (y^2, xy-a_1yz-a_2yt, t^2-*yz-*yt, zt-a_3yz-a_4yt&, \\xt-*yz-*yt, z^2-*yz-*yt, xz-*yz-*yt&).\end{split}\] Applying the change of coordinates \[\left\{ \begin{array}{lll} x & \mapsto & x+a_1z+a_2t \\ y & \mapsto & y \\ z & \mapsto & z+a_4y \\ t & \mapsto & t+a_3y \end{array} \right.\] we get \[J = (y^2, xy, zt, t^2-*yz-*yt, xt-*yz-*yt, z^2-*yz-*yt, xz-byz-*yt)\] and sending $x$ into $x+by$ we obtain
\[(y^2, xy, zt, xz-*yt, t^2-*yz-*yt, z^2-*yz-*yt, xt-*yz-*yt)\] as claimed.
\item \label{case 11} $z > t > x > y$: \[\begin{split}h(I) = (y^2, xy, xt-a_1x^2-*yt-*yz, xz-a_2x^2-*yt-*yz&, \\t^2-*x^2-*yt-*yz, zt-*x^2-*yt-*yz, z^2-*x^2-*yt-*yz&).\end{split}\] Sending $t$ into $t+a_1x$ and $z$ into $z+a_2x$ we get \[\begin{split}J = (y^2, xy, xt-*yt-*yz, xz-*yt-*yz, t^2-b_1x^2-*yt-*yz&, \\zt-b_2x^2-*yt-*yz, z^2-b_3x^2-*yt-*yz&).\end{split}\] Imposing the presence of a point $\lbrack x_0, y_0, z_0, t_0\rbrack$ in $\mathcal{V}(J)$ we get that the coefficients $b_1$, $b_2$ and $b_3$ must vanish. Therefore \[\begin{split}J = (y^2, xy, xt-*yt-*yz, xz-*yt-*yz, t^2-*yt-*yz&, \\zt-*yt-*yz, z^2-*yt-*yz&)\end{split}\] and we have come back to case $\ref{case 3}$.

\item \label{case 12} $z > t > y > x$: \[\begin{split}h(I) = (xy-a_1x^2, xt-a_2x^2, y^2-*x^2, xz-a_3x^2-*yt&, \\t^2-*x^2-*yt-*yz, zt-*x^2-*yt-*yz, z^2-*x^2-*yt-*yz&).\end{split}\] Sending $y$ into $y+a_1x$, $z$ into $z+a_3x$ and $t$ into $t+a_2x$ we get \[\begin{split}J = (xy, xt, y^2-b_1x^2, xz-*yt, t^2-b_2x^2-*yt-*yz&, \\zt-b_3x^2-*yt-*yz, z^2-b_4x^2-*yt-*yz&).\end{split}\] Imposing the presence of a point $\lbrack x_0, y_0, z_0, t_0\rbrack$ in $\mathcal{V}(J)$ we get that the coefficients $b_1$, $b_2$, $b_3$ and $b_4$ must vanish. Therefore \[J = (xy, xt, y^2, xz-*yt, t^2-*yt-*yz, zt-*yt-*yz, z^2-*yt-*yz)\] and we have come back to case $\ref{case 3}$.

\item \label{case 1} $x > y > z > t$: \[\begin{split}h(I) = (t^2, zt, z^2-2ayt, xt-*yz-*yt, xz-*yz-*yt&, \\y^2-2byz-2cyt, xy-*yz-*yt&).\end{split}\] Sending $y$ into $y + bz+ (ab^2+c)t$ we obtain \[(t^2, zt, y^2, z^2-*yt, xt-*yz-*yt, xz-*yz-*yt, xy-*yz-*yt)\] and therefore we are back to case \ref{case 3}. 

\item \label{case 2} $x > z > y > t$: \[\begin{split}h(I) = (t^2, y^2-2ayt, zt-*yt, xt-*yz-*yt, xy-*yz-*yt&, \\z^2-*yz-*yt, xz-*yz-*yt&).\end{split}\] Sending $y$ into $y+at$ we get back to case \ref{case 3}.

\item \label{case 9} $z > y > x > t$: \[\begin{split}h(I) = (t^2, xt, zt-*x^2-*yt, xy-*x^2-ayt, xz-*x^2-*yt&, \\y^2-*x^2-2byt, z^2-*x^2-*yt-*yz&).\end{split}\] Sending $x$ into $x+at$ and $y$ into $y+bt$ we get back to case \ref{case 12}.

\item \label{case 10} $z > y > t > x$: \[\begin{split}h(I) = (xt-*x^2, xy-*x^2, t^2-*x^2, xz-*x^2-*yt, zt-*x^2-*yt&, \\y^2-*x^2-2ayt, z^2-*x^2-*yt-*yz&).\end{split}\] Sending $y$ into $y+at$ we get back to case \ref{case 12}.

\item \label{case 7} $z > x > y > t$: \[\begin{split}h(I) = (t^2, xt-ayt, y^2-2byt, xy-cyt, zt-*x^2-*&yt, \\xz-*x^2-*yz-*yt, z^2-*x^2-*yz-*&yt).\end{split}\] Sending $x$ into $\displaystyle x+(c-ab)t$ and $y$ into $y+bt$ we get \[(t^2, y^2, xy, xt-*yt, zt-*x^2-*yt, xz-*x^2-*yz-*yt, z^2-*x^2-*yz-*yt)\] and hence we have come back to case \ref{case 11}.

\item \label{case 8} $z > x > t > y$: \[\begin{split}h(I) = (y^2, xy-ayt, t^2-*yz-*yt, xt-*yz-*yt, zt-*x^2-*yz-*yt&, \\xz-*x^2-*yz-*yt, z^2-*x^2-*yz-*yt&).\end{split}\] Sending $x$ into $x+at$ we get back to case \ref{case 11}.

\item \label{case 4} $y > x > z > t$: \[\begin{split}h(I) = (t^2, zt, xt, z^2-*yt, xz-*yt, xy-*x^2-ayz-*yt&, \\y^2-*x^2-2byz-*yt&).\end{split}\] Sending $x$ into $x+az$ and $y$ into $y+bz$ we get back to case \ref{case 9}.

\item \label{case 5} $y > z > x > t$: \[\begin{split}h(I) = (t^2, xt, zt-*x^2, xz-*x^2-*yt, xy-*x^2-*yt&, \\z^2-*x^2-*yt, y^2-2ayz-*x^2-*yt&).\end{split}\]
Sending $y$ into $y+az$ we get back to case \ref{case 9}.

\item \label{case 6} $y > z > t > x$: \[\begin{split}h(I) = (xt-*x^2, xz-*x^2, xy-*x^2, t^2-*x^2, zt-*x^2&, \\z^2-*yt-*x^2, y^2-2ayz-*yt-*x^2&).\end{split}\] Sending $y$ into $y+az$ we get back to case \ref{case 10}. \qedhere
\end{enumerate}
\renewcommand{\labelenumi}{(\roman{enumi})}
\end{proof}

\section*{Acknowledgements}
The author would like to thank his advisor Aldo Conca for introducing him to the problem studied in the present article and for a great deal of valuable comments and suggestions. The author also wishes to thank Ralf Fr\"oberg and Matteo Varbaro for useful discussions and Ezra Miller for pointing out the reference containing the definition of unital binomial ideal.

	\phantomsection
	\addcontentsline{toc}{chapter}{\bibname}
	\printbibliography
\end{document}